\documentclass{article}

\usepackage{amsmath,amsopn,amssymb,epsfig,amsthm}
\usepackage{color}
\usepackage{multirow}

\definecolor{amethyst}{rgb}{0.6, 0.4, 0.8}
\definecolor{orange}{rgb}{1,0.5,0}

\newtheorem{Algorithm}{Algorithm}[section]
\newtheorem{Theorem}{Theorem}[section]
\newtheorem{Lemma}{Lemma}[section]

\newtheorem{Remark}{Remark}[section]

\newtheorem{Corollary}{Corollary}[section]
\newcommand{\bb}{\begin{bmatrix}}
\newcommand{\eb}{\end{bmatrix}}
\newcommand{\bl}[1]{\begin{list}{#1}{\usecounter{bean}}} \newcommand{\el}{\end{list}}
\newcommand{\bel}[1]{\begin{equation} \label{#1}} \newcommand{\eel}{\end{equation}}

\def\r2n2n{\mathbb{R}^{2n\times 2n}}
\def\c2n2n{\mathbb{C}^{2n\times 2n}}

{\color{blue}
\title{
An iterative method for solving the stable subspace of a matrix pencil and its application.
}
}

\author{Matthew M. Lin
\thanks{Department of Mathematics, National Cheng Kung University, Tainan 701, Taiwan. The first author was supported by the Ministry of Science and Technology of Taiwan under grant 104-2115-M-006-017-MY3. {\tt (mhlin@mail.ncku.edu.tw)}}
\and
Chun-Yueh Chiang
\thanks{Corresponding Author, Center for General Education, National Formosa University, Huwei 632, Taiwan. The second author was supported by the Ministry of Science and Technology of Taiwan under grant 105-2115-M-150-001. {\tt (chiang@nfu.edu.tw)}}
}

\begin{document}
\maketitle

\begin{abstract}
This work is to propose an iterative method of choice to compute a stable subspace of a regular matrix pencil. This approach is to define a sequence of matrix pencils via
particular left null spaces. We show that this iteration preserves a discrete-type flow depending only on the initial matrix pencil. Via this recursion relationship,  we propose an accelerated iterative method to compute the stable subspace and use it to provide a theoretical result to solve the principal square root of a given matrix, both nonsingular and singular. We show that this method can not only find out the matrix square root, but also construct an iterative approach which converges to the square root with any desired order.
\end{abstract}


\pagestyle{myheadings} \thispagestyle{plain}

\section{Introduction}
%
%
Throughout this paper we shall use the following notation to facilitate our discussions. $\lambda(A)$ and $\lambda(A,B)$ denote the sets of eigenvalues of the matrix $A$ and the matrix pencil $A-\lambda B$, respectively, and let $\rho(A)$ be the spectral radius of the square matrix $A$. $\mathbb{C}^+$  and $\mathbb{C}^-$ represent the open right and left half complex planes.

Given a regular $n\times n$ matrix pencil $A- \lambda B$ (i.e., $\det(A-\lambda B)$ is not identically zero for all $\lambda$) and an integer $m\leq n$, we want to find in this work a full rank matrix $U\in\mathbb{C}^{n\times m}$ such that
\begin{equation}\label{eq:gen}
A U = B U  \Lambda,
\end{equation}
{where $\Lambda\in\mathbb{C}^{m\times m}$ and $\rho(\Lambda) < 1$.}

Note that the column space $\mathcal{U}=\mbox{Span}\{U\}$ is called the stable deflating subspace of $A-\lambda B$. Specially,
 $\mathcal{U}$ is called the stable invariant space if $B$ is the identity matrix.
Over the past few decades, considerable attention has been paid to
study the property of the invariant and deflating subspace~\cite{Gohberg06}. In application, one can obtain
the solutions of algebraic Riccati-type matrix equations by computing its corresponding  stable deflating subspaces or stable invariant subspaces, e.g., \cite{Lancaster95,Bini2012}. Particularly, this problem is related to the so-called \emph{generalized spectral divide and conquer} (SDC) problem~\cite{Bai1997,Bini2012}, which is to find a pair of left and right deflating subspaces $\mathcal{L}$ and $\mathcal{R}$ such that
\begin{equation*}
A\mathcal{R} \subset \mathcal{L},\quad B\mathcal{R} \subset \mathcal{L},
\end{equation*}
corresponding to eigenvalues of the pair $A-\lambda B$ in a specified region
$\mathcal{D}\subset \mathbb{C}$. That is, find two nonsingular partitioned matrices $U_L = \bb U_{L_1},U_{L_2}\eb$ and $U_R= \bb U_{R_1},U_{R_2}\eb$ with $\mathcal{L} = \mbox{span}(U_{L_1})$ and
$\mathcal{R} = \mbox{span}(U_{R_1})$ so that
\begin{equation*}
A U_R = U_L\bb A_{11} & A_{12} \\0 & A_{22}\eb, \quad
B U_R = U_L\bb B_{11} & B_{12} \\0 & B_{22}\eb, \quad
\end{equation*}
and the eigenvalues of $A_{11} -\lambda B_{11}$ are the eigenvalues of $A-\lambda B$ in  the region $\mathcal{D}$.
We notice that
if $A_{11} - \lambda B_{11}$ has no infinite eigenvalues, then $B_{11}$ is invertible and
\begin{equation*}
AU_{R_1} = BU_{R_1} (B_{11}^{-1} A_{11});
\end{equation*}
if $A_{11} - \lambda B_{11}$ has no zero eigenvalues, then
$A_{11}$ is invertible and
\begin{equation*}
AU_{R_1} ( A_{11}^{-1} B_{11})= BU_{R_1}.
\end{equation*}

Note that the region $\mathcal{D}$ in the SDC problem is generally assumed in the interior (or exterior) of the unit disk. Otherwise, the M\"{o}bius transformations $(\alpha A + \beta B) (\gamma A + \delta B)^{-1}$ can be applied to transform original region as a rather general region~\cite{Bai1997}.

One direct method to solve~\eqref{eq:gen} (not requires $\rho(\Lambda)<1$), is to apply the so-called QZ algorithm.
%
That is, through the QZ algorithm, the matrix $A$ is reduced to triangular or upper quasi-triangular form and $B$ to upper triangular form. One is then able to compute eigenvectors through the reduced form (see~\cite{Demmelbook00, Golub2013,Kressnerbook05} for the details).
Unlike the direct method, we propose in this work an iterative method, AB-algorithm, to solve~\eqref{eq:gen}.
This method is done by defining
 a sequence of matrix pencils $\{A_k-\lambda B_k\}$
with $(A_1, B_1) = (A,B)$ and
$(A_k, B_k) = ( \mathcal{M}_{k-1} A_{k-1},  \mathcal{N}_{k-1} B_{1})$ for any integer $k>1$. Here,
$( \mathcal{N}_{k-1}, \mathcal{M}_{k-1})$ is a solution belonging to the left null space of $ \left[\begin{array}{c}A_1 \\-B_k\end{array}\right]
$, that is,
\begin{equation}\label{eq:null}
\mathcal{N}_{k} A_{1} = \mathcal{M}_{k} B_{k}.
\end{equation}
{
 Due to the specific structure embedded in the matrix pencil $A_1-\lambda B_k$, some $( \mathcal{N}_{k}, \mathcal{M}_{k})$ can be designed such that $\{A_k-\lambda B_k\}$ have the same structure.
 We refer the reader to \cite{WWL2006,Bini2012} and to the references therein. In these works, iterative algorithms for computing invariant subspace of structured matrix pencil $A_1-\lambda B_k$  are provided with a quadratic convergence for solving algebraic Riccati and related matrix equations.
%
%
 }

Observe that $A_1 U = B_1 U \Lambda^{1}$.
Suppose this process can be continually iterated to obtain
the new pencil $A_k-\lambda B_k$ such that $A_k U = B_k U \Lambda^k$. It must be that
\begin{eqnarray}\label{eq:eig1}
 A_{k+1} U &=&  \mathcal{M}_kA_k U=  \mathcal{M}_k B_k U  \Lambda^k\nonumber
 \\
 &=&  \mathcal{N}_k A_1 U \Lambda^k =  \mathcal{N}_k B_1 U \Lambda^{k+1} = B_{k+1} U \Lambda^{k+1}.
\end{eqnarray}

This implies that if $\rho (\Lambda)< 1$, and if the sequence $\{B_k\}$  is uniformly bounded, then $\lim\limits_{k\rightarrow \infty} A_k U = 0$. Once the sequence $\{A_k\}$ also converges, say $A_\infty := \lim\limits_{k\rightarrow \infty} A_k$, we are able to solve the solution $U$ by computing the the right null space of $A_\infty$.
{We notice that
this AB-algorithm theoretically not only preserves a discrete-type flow property (see Theorem~\ref{thm:trans1}) but can be accelerated with the rate of convergence of any desired order.
To demonstrate the capability of this algorithm, we use it to provide a theoretical result to compute the matrix square root as an example. It is known that matrix square root is not unique (even up to sign) or even exists, for example,
\begin{equation*}
\left[\begin{array}{cc} 1 & 0 \\0 & 1\end{array}\right]
= \left[\begin{array}{cc}
\cos(\theta)
& \sin(\theta) \\\sin(\theta) & -\cos(\theta)\end{array}\right]^2
\end{equation*}
for any $\theta\in\mathbb{R}$ and $\left[\begin{array}{cc} 0 & 1 \\0 & 0\end{array}\right] $ does not have a square root. Indeed,  let $S \in\mathbb{C}^{n\times n}$ be a matrix having no nonpositive real eigenvalues. Then the quadratic matrix equation
\begin{equation}\label{eq:quad}
X^2 - S =0
\end{equation}
has a unique solution $X$ such that $\lambda(X) \subset \mathbb{C}^+$. This is called the principal square root of $S$ and denote it by $\sqrt{S}$~\cite{Deprima1974, Higham1987}. Numerical methods for computing the matrix principal square root, including the (modified) Schur method, Newton's method, and its variants, have been widely discussed in the numerical linear algebra community. See~\cite{Laasonen1958, Deprima1974, Higham1986,Higham1987,Higham1997,Meini2004,Highambook08,Mizuno2016} and the references therein.
Unlike conventional methods, this AB-algorithm can be modified to obtain the square root efficiently, that is, the rate of convergence of this algorithm can be of any desired order $r$. Specifically, our method is equivalent to the so-called Newton's method when $r=2$. More precisely, though preserving a similar convergence property like the Newton's method, our algorithm can be shown that, under mild adjustments, the speed of convergence can be q-superlinearly with any order~\cite{Kelly1995}.
%

This work is organized as follows.  In section 2 we provide properties of the AB-Algorithm. In section
3 we modified this AB-Algorithm so that its convergence can be of any order. In section 4 we report a
numerical application to solve the matrix square root,
and the concluding remarks are given in section 5.

%
%


\section{The AB-Algorithm and Its Corresponding Properties}

Recall that the idea of the AB-Algorithm depends heavily on the determination of the left null space of $ \left[\begin{array}{c}A_1 \\-B_k\end{array}\right]$.
Observe that
$
B_1(A_1+B_1)^{-1}A_1 - A_1(A_1+B_1)^{-1}B_1 =  0
$, if $-1\not\in\lambda(A_1,B_1)$ and $(A_1, B_1)$ are two $n\times n$ matrices. That is, $(N_1, M_1):=(B_1(A_1+B_1)^{-1},A_1(A_1+B_1)^{-1})$ is the left null space of $ \left[\begin{array}{c}A_1 \\-B_1\end{array}\right]$. Using the same procedure, we would like to generate the matrix sequences $\{A_k\}$ and $\{B_k\}$ by defining
\begin{subequations}\label{signal2}
\begin{align}
A_k&= A_1(A_1+B_{k-1})^{-1}A_{k-1}, 
\\
B_k&=B_{k-1}(A_1+B_{k-1})^{-1}B_1,
\end{align}
\end{subequations}
once the process can be iterated.

It should be noted that if $ \mathcal{M}_{k-1} = B_{k-1}(A_1+B_{k-1})^{-1}$ and $
 \mathcal{N}_{k-1} = A_1(A_1+B_{k-1})^{-1}$ for any integer $k>1$, it can be seen  that
 $B_{k-1}(A_1+B_{k-1})^{-1} A_1 = A_1 (A_1+B_{k-1})^{-1} B_{k-1}$, which satisfies the assumption~\eqref{eq:null}.
For simplicity,  we let $\Delta_{i,j}:=(A_{i}+B_{j})^{-1}$ so that the sequences $\{A_k\}$ and $\{B_k\}$ in~\eqref{signal2} can be rewritten as
  \begin{subequations}\label{eq1}
  \begin{align}
  A_k&=A_1\Delta_{1,k-1}A_{k-1}=A_{k-1}-B_{k-1}\Delta_{1,k-1}A_{k-1},\label{eq1a}\\
  B_k&=B_{k-1}\Delta_{1,k-1}B_{1}=B_{1}-A_{1}\Delta_{1,k-1}B_{1}.\label{eq1b}
  \end{align}
  \end{subequations}
 Based on~\eqref{eq1}, we propose the following AB-algorithm for computing the stable {subspace} of  the matrix pencil {$A_1-\lambda B_1$}:

\begin{Algorithm}\label{abalg}
{\emph{(AB-Algorithm)}}
\begin{enumerate}

\item Given a pencil $A_1-\lambda B_1$, initialize a tolerance  $\tau>0$ and a positive integer $kmax$.

\item For $k = 2...$, iterate {until
{$\mbox{dist}(\mbox{Null}(A_{k-1}),\mbox{Null}(A_{k}))< \tau$} or {$k>kmax$}.}
\begin{enumerate}
\item $A_k = A_1\Delta_{1,k-1}A_{k-1}$,
\item $B_k = B_{k-1}\Delta_{1,k-1}B_1$,
\end{enumerate}
\end{enumerate}
{
Here, ``Null$(\cdot)$'' denotes the null space of the given matrix and ``dist$(\cdot,\cdot)$''denotes the distance between two subspaces~\cite{Golub2013}.}
\end{Algorithm}
%



Note that on the one hand, Algorithm~\ref{abalg} provides an alternative approach for finding the stable invariant subspace $U$ (i.e., $A_1 U = U\Lambda$ and $\rho(\Lambda) < 1$) of the matrix $A_1$ by constructing $A_\infty$ (once it exists) directly as follows:
%
\begin{Remark}
If no breakdown occurs in Algorithm~\ref{abalg}
and $B_1=I_n$, for any integer $k>1$ we have
\begin{subequations}
\begin{align}
A_k &=A_1^k(
\sum_{j= 0}^{k-1} A_1^j
)^{-1},\label{s1}\\
B_k &=(\sum_{j= 0}^{k-1} A_1^j
)^{-1}.\label{s2}
\end{align}
\end{subequations}
In other words, to obtain the stable subspace of the matrix $A_1$, we only need to focus on the {iterations} generated by \eqref{s1}.
\end{Remark}

%
%
%
On the other hand, once the iteration is available, we are interested in characterizing the transformation of eigenvalues of the matrix pencil $A_1-\lambda B_1$ after each iteration.
 First, we give an observation about the relationship between the eigenvalues of $A_k-\lambda B_k$ and the eigenvalues of $A_1-\lambda B_1$. Since the proof can be read off from~\eqref{eq:eig1}, we omit our proof here.
\begin{Lemma}\label{lem:eig}
Let $A_1-\lambda B_1$ be a regular matrix pencil, and let
$\{A_k-\lambda B_k\}$ be the sequence of matrix pencils generated by Algorithm~\ref{abalg}, if no breakdown occurs.
If $\lambda\in\lambda(A_1,B_1)$ with $\lambda\in\mathbb{C}\cup\{\infty\}$, then $\lambda^k\in\lambda(A_k,B_k).\, (Here, \infty^k:=\infty)$
\end{Lemma}

%
Subsequently, we have the following theorem which gives rise to the
appearance of new eigenvalues induced by the AB-algorithm.
{

\begin{Theorem}\label{Lem:eig}
Let $A_1-\lambda B_1$ be a regular matrix pencil, and let
$\{A_k-\lambda B_k\}$ be the sequence of matrix pencils generated by Algorithm~\ref{abalg}, if no breakdown occurs.
Let  $\{{\lambda}_1^{(i,k)},\ldots,{\lambda}_n^{(i,k)}\}$
be the set of eigenvalues of the matrix pencils $A_{i}-\lambda B_{k}$ for any two positive integers $i$ and $k$. Then,
for $1\leq j \leq n$,
the set of eigenvalues has the following properties:
\begin{enumerate}
\item
$
\lambda_j^{(1,k)}=\left\{
                 \begin{array}{rl}
          \sum\limits_{s=1}^{k}(\lambda_j^{(1,1)})^s,&\lambda_j^{(1,1)}\in\mathbb{C}, \\
                    \infty,&\lambda_j^{(1,1)}=\infty.\\
                 \end{array}
 \right.
$
\item
$
{\lambda}_j^{(i,1)}=\left\{
                 \begin{array}{rl}
          \dfrac{(\lambda_j^{(1,1)})^i}{\sum\limits_{s=0}^{i-1}(\lambda_j^{(1,1)})^s},&\lambda_j^{(1,1)}\in\mathbb{C}, \\
                    \infty,&\lambda_j^{(1,1)}=\infty.\\
                 \end{array}
 \right.
$
%
%
\item 
$
{\lambda}_j^{(i,k)}=\left\{
                 \begin{array}{rl}
          (\lambda_j^{(1,1)})^i\dfrac{\sum\limits_{s=0}^{k-1}(\lambda_j^{(1,1)})^s}{\sum\limits_{s=0}^{i-1}(\lambda_j^{(1,1)})^s},&\lambda_j^{(1,1)}\in\mathbb{C}, \\
                    \infty,&\lambda_j^{(1,1)}=\infty.\\
                 \end{array}
 \right.
 $
\end{enumerate}
\end{Theorem}

%

\begin{proof}
Assume without loss of generality that $A_1$ and $B_1$ are upper triangular matrices. Otherwise, let $U$ and $V$ be two unitary matrices such that $U^H A_1 V$ and $U^H B_1 V$ both are upper triangular matrices. Upon using~\eqref{signal2},  it can be seen that $A_k$ and $B_k$ are also upper triangular, and
 \begin{eqnarray*}
\lambda_j^{(1,k)}&=&\left\{
                 \begin{array}{rl}
          (1+\lambda_j^{(1,k-1)})\lambda_j^{(1,1)},&\lambda_j^{(1,1)}\in\mathbb{C}, \\
                    \infty,&\lambda_j^{(1,1)}=\infty,\\
                 \end{array}
 \right.\\
%
 {\lambda}_j^{(i,1)} &=&\left\{
                 \begin{array}{rl}
          \dfrac{{\lambda}_j^{(i,i)}}{
          1+{\lambda}_j^{(1,i-1)}},&\lambda_j^{(1,1)}\in\mathbb{C}, \\
                    \infty,&\lambda_j^{(1,1)}=\infty.\\
                 \end{array}
 \right.
\end{eqnarray*}
Moreover,
  \begin{equation*}
 {\lambda}_j^{(i,k)}=\left\{
                 \begin{array}{rl}
          (1+{\lambda}_j^{(1,k-1)}){\lambda}_j^{(i,1)},&\lambda_j^{(1,1)}\in\mathbb{C}, \\
                    \infty,&\lambda_j^{(1,1)}=\infty,\\
                 \end{array}
 \right.
\end{equation*}%
for $i, k\geq 2$. We remark that $\lambda_j^{(1,i-1)}\neq -1$ since $A_i-\lambda B_i$ is well-defined, and from Lemma~\ref{lem:eig}, we have $\lambda_j^{(i,i)}=(\lambda_j^{(1,1)})^i$, which completes the proof of the theorem.

\end{proof}
}

We notice that Algorithm~\ref{abalg} is workable { if and only if} the sum of matrices $A_1$ and $B_{k-1}$, for any integer $k>1$, is invertible, that is, {$-1\not\in\lambda(A_1,B_{k-1})$}, for any integer $k>1$.
 This capacity can be completely characterized by the
$p$th roots of unity, except itself.

%
\begin{Theorem}\label{thm:well-defined}
Let {$A_1-\lambda B_1$} be a regular matrix pencil, and let
\begin{equation*}
S_k= \bigcup\limits_{2\leq p\leq k+1}
\{
e^{\frac{2q\pi i}{p}} : {1\leq q\leq p-1}
\}.
\end{equation*}
If
\begin{equation*}
S_k \cap \lambda(A_1,B_{1}) = \phi,
\end{equation*}
then the sequence of matrix pencils
{$A_k-\lambda B_k$}, for any integer $k\geq 1$, can be generated using Algorithm~\ref{abalg}, or, generally,
all sequences of matrices $\{A_k-\lambda B_k\}$ generated by iterations \eqref{signal2} with the initial matrix pencil $A_1-\lambda B_1$ are no breakdown, if
{
\begin{align}\label{cond}
S_\infty\cap\lambda(A_1,B_{1})=\phi.
\end{align}}

\end{Theorem}

%
%
%

%

%
%

\begin{Corollary}\label{cor:dif}
For any positive integers $i,j$ and $k$, we have $A_k-B_k=A_1-B_1$, that is, $A_i-A_j=B_i-B_j$, provided that {$S_{\max\{i,j,k\}}\cap\lambda(A_1,B_{1})=\phi$}.

\end{Corollary}

\begin{proof}
The proof is by induction on $k$. When $k=1$, the result is evident. Suppose we have proved this corollary for $k =\ell$. Then, by the induction hypothesis
\begin{align*}
A_{\ell+1}-B_{\ell+1}&=A_\ell-B_\ell\Delta_{1,\ell}A_\ell-B_\ell\Delta_{1,\ell}B_1\\
&=A_\ell-B_\ell\Delta_{1,\ell}(A_\ell+B_1)=A_\ell-B_\ell=A_1-B_1.
\end{align*}


This completes the proof.
\end{proof}
{
From Corollary~\ref{cor:dif}, each step of $B_k$ can be obtained by $B_k=A_k+B_1-A_1$.
We conclude that the counts of Algorithm~\ref{abalg} for one iteration is $\frac{14}{3}n^3$ flops. This is because the computation is preliminary determined by the product of two $n\times n$ matrices, the calculation of the Gaussian elimination with partial pivoting, and the performance of solving $n$ lower triangular systems and $n$ upper triangular systems. Hence, the calculation of the counts contains a PLU factorization $(\mbox{cost}:\frac{2}{3}n^3$ flops) and two multiplication $(\mbox{cost}:4n^3$ flops).
Here, we ignore any $O(n^2)$ operation counts and the memory counts. We notice that the computational cost of QZ algorithm is about $46n^3$ flops (the right eigenvectors are desired). On the other hand, it follows from Theorem~\ref{thm:well-defined} that Algorithm~\ref{abalg} is well-defined, once~\eqref{cond} is satisfied.  Here, we use Gaussian elimination with partial pivoting, which is known to perform well and usually eliminate the numerical instability in practice~\cite{WICS:WICS164}, to compute the matrix inverse so that the iteration will not terminate prematurely. To perform the error analysis and decide the numerical stability of Algorithm~\ref{abalg}, the reader is referred to~\cite{HUANG20091452} for a similar discussion.
}

We remark that Corollary~\ref{cor:dif} also implies that
$\lim\limits_{k\rightarrow\infty} A_k$ exists if and only if $\lim\limits_{k\rightarrow\infty} B_k$ exists. Note that in~\eqref{signal2}, the iterations of {the matrix pencils $A_k-\lambda B_k$, for $k\geq 1$, are relative to the initial pencil $A_1-\lambda B_1$}. We would like to derive a more general iterative method, which are easily accessible through any initial pencil {$A_i-\lambda B_i$}.  To this purpose, we shall first introduce the well-known Sherman Morrison Woodbury formula (SMWF).
\begin{Lemma}\cite{Bernstein2005}\label{Schur}
Let $A$ and $B$ be two arbitrary matrices of size $n$, and  let $X$ and $Y$ be two $n\times n$ nonsingular matrices. Assume that $Y^{-1}\pm BX^{-1}A$ is nonsingular. Then, $X\pm A Y B$ is invertible and
 \[
(X\pm A Y B)^{-1}=X^{-1}\mp X^{-1}A(Y^{-1}\pm B X^{-1}A)^{-1}BX^{-1}.
\]
\end{Lemma}


This lemma gives a useful method to prove the following result.

\begin{Theorem}\label{thm:trans1}
Let the assumption~\eqref{cond} holds and {$\{A_k-\lambda B_k\}$} be the sequence of matrix pencils obtained by~\eqref{signal2} with initial {$A_1-\lambda B_1$}. Then,
\begin{subequations}\label{tran2}
\begin{align}
A_{i+j}&=A_{i}(A_{i}+B_{j})^{-1}A_{j},\\
B_{i+j}&=B_{j}(A_{i}+B_{j})^{-1}B_{i},
\end{align}
where $i$ and $j$ are any two positive integers.
\end{subequations}
\end{Theorem}
\begin{proof}
This proof is divided into two parts.
We first fix $j=1$ and show that the statement~\eqref{tran2} is true for any positive integer $i$.
We prove by induction on $i$. When $i =1$, the statement~\eqref{tran2} is definitely true from the definition of $A_2$ and $B_2$. Suppose~\eqref{tran2} is true for $i=s$.  It follows from Lemma~\ref{Schur} that
%
\begin{align*}
\Delta_{1,s+1}
&=(A_1+B_s-A_s\Delta_{s,1}B_s)^{-1}\\
&=\Delta_{1,s}+\Delta_{1,s}A_{s}(A_s+B_1-B_s\Delta_{1,s}A_s)^{-1}B_{s}\Delta_{1,s}\\
&=\Delta_{1,s}+\Delta_{1,s}A_{s}\Delta_{1+s,1}B_{s}\Delta_{1,s},\\
\Delta_{1+s,1} &=(A_s-B_s\Delta_{1,s}A_s +B_1)^{-1}\\
&=\Delta_{s,1}+\Delta_{s,1}B_{s}(A_1+B_s-A_s\Delta_{s,1}B_s)^{-1}A_{s}\Delta_{s,1}\\
&=\Delta_{s,1}+\Delta_{s,1}B_s\Delta_{1,s+1}A_s\Delta_{s,1}.
%
\end{align*}
Thus, we have
 \begin{align*}
A_{(s+1)+1}&=A_{1+(s+1)}=A_{s+1}-B_{s+1}\Delta_{1,s+1}A_{s+1}\\
&=A_1-B_{1}\left[\Delta_{s,1}+\Delta_{s,1}B_s\Delta_{1,s+1}A_s\Delta_{s,1}\right]A_{1}\\
&=A_{1}-B_{1}\Delta_{s+1,1}A_{1}=A_{s+1}\Delta_{s+1,1}A_{1},\\
B_{(s+1)+1}&=B_{1+(s+1)}=B_{1}-A_{1}\Delta_{1,s+1}B_{1}\\
&=B_{1}-A_{1}\left[\Delta_{1,s}+\Delta_{1,s}A_{s}\Delta_{s+1,1}B_{s}\Delta_{1,s}\right]B_{1}\\
&=B_{s+1}-A_{s+1}\Delta_{s+1,1}B_{s+1}=B_{1}\Delta_{s+1,1}B_{s+1},
\end{align*}
which completes the proof of the first part.

Now suppose that \eqref{tran2} is true for $j=s$ and any $i$. In particular,
\begin{align*}
\Delta_{i,s+1}
&=(A_i+B_s-A_s\Delta_{s,1}B_s)^{-1}\\
&=\Delta_{i,s}+\Delta_{i,s}A_{s}(A_s+B_1-B_s\Delta_{i,s}A_s)^{-1}B_{s}\Delta_{i,s}\\
&=\Delta_{i,s}+\Delta_{i,s}A_{s}\Delta_{i+s,1}B_{s}\Delta_{i,s},\\
\Delta_{i+s,1} &= (A_s-B_s\Delta_{i,s}A_s +B_1)^{-1}\\
&=\Delta_{s,1}+\Delta_{s,1}B_{s}(A_i+B_s-A_s\Delta_{s,1}B_s)^{-1}A_{s}\Delta_{s,1}\\
&=\Delta_{s,1}+\Delta_{s,1}B_s\Delta_{i,s+1}A_s\Delta_{s,1}.
\end{align*}
This implies
 \begin{align*}
A_{i+(s+1)}&=A_{(i+s)+1}=A_{1}-B_{1}\Delta_{i+s,1}A_{1}\\
&=A_1-B_{1}\left[\Delta_{s,1}+\Delta_{s,1}B_s\Delta_{i,s+1}A_s
\Delta_{s,1}
\right]A_{1}\\
&=A_{s+1}-B_{s+1}\Delta_{i,s+1}A_{s+1}=A_{i}\Delta_{i,s+1}A_{s+1},\\
B_{i+(s+1)}&=B_{(i+s)+1}=B_{i+s}-A_{i+s}\Delta_{i+s,1}B_{i+s}\\
&=B_{i}-A_{i}\left[\Delta_{i,s}+\Delta_{i,s}A_{s}\Delta_{i+s,1}B_{s}\Delta_{i,s}\right]B_{i}\\
&=B_{i}-A_{i}\Delta_{i,s+1}B_{i}=B_{s+1}\Delta_{i,s+1}B_{i},
\end{align*}
which completes the proof of the theorem.
\end{proof}

{Two things are required to be noted. First,
Theorem~\ref{thm:trans1} implies that the iterative sequence $\{A_k-\lambda B_k\}$ can be formulated explicitly from any two matrix pencils
$A_i-\lambda B_i$ and $A_j-\lambda B_j$, where $i+j = k$. The formula also gives rise to {
a discrete-type flow} and can be used to accelerate the iterations given in Algorithm~\ref{abalg}.
Second, it follows} from  Corollary~\ref{cor:dif} and Theorem~\ref{thm:trans1}  that $A_k=A_{k-1}\Delta_{1,k-1} A_{1}=A_{1}\Delta_{1,k-1} A_{k-1}=A_{k-1}\Delta_{k-1,1} A_{1}=A_{1}\Delta_{k-1,1} A_{k-1}$. It shows that the iterations $A_k$ and $B_k$, regardless of the assumptions~\eqref{signal2}, have the following four equivalent forms by using the same initial matrix pencil:
{
\begin{equation*}
\begin{array}{|c|c|}
\hline
\multirow{2}{*}{1.} & A_k^{(1)}=A_1^{(1)}(A_1^{(1)}+B_{k-1}^{(1)})^{-1}A_{k-1}^{(1)},\\
& B_k^{(1)}=B_{k-1}^{(1)}(A_1^{(1)}+B_{k-1}^{(1)})^{-1}B_1^{(1)};
  \\
\hline
\multirow{2}{*}{2.} & A_k^{(2)}=A_1^{(2)}(B_1^{(2)}+A_{k-1}^{(2)})^{-1}A_{k-1}^{(2)},\,\\
& B_k^{(2)}=B_{k-1}^{(2)}(A_1^{(2)}+B_{k-1}^{(2)})^{-1}B_1^{(2)};
  \\
\hline
\multirow{2}{*}{3.} & A_k^{(3)}=A_1^{(3)}(A_1^{(3)}+B_{k-1}^{(3)})^{-1}A_{k-1}^{(3)},\\
& B_k^{(3)}=B_{k-1}^{(3)}(B_1^{(3)}+A_{k-1}^{(3)})^{-1}B_1^{(3)};
  \\
\hline
\multirow{2}{*}{4.} & A_k^{(4)}=A_1^{(4)}(B_1^{(4)}+A_{k-1}^{(4)})^{-1}A_{k-1}^{(4)},\\
& B_k^{(4)}=B_{k-1}^{(4)}(B_1^{(4)}+A_{k-1}^{(4)})^{-1}B_1^{(4)}.\\
\hline
\end{array}
\end{equation*}
}

The next theorem is to know how the eigeninformation is transferred during the iterative process.
\begin{Theorem}\label{thm:trans2}
Let $A_1-\lambda B_1$ be a regular matrix pencil, and let
$\{A_k-\lambda B_k\}$ be the sequence of matrices generated by Algorithm~\ref{abalg}. Suppose that
the condition~\eqref{cond} holds and
$A_1U=B_1U\Lambda$.
Then,
\begin{itemize}
\item [\emph{(a)}] $A_1 U= B_k U \sum\limits_{j=1}^k \Lambda^j$.
\item [\emph{(b)}] $A_k U= B_k U \Lambda^k$. In particular, if $1\not\in\lambda(\Lambda)$, then
\begin{equation}\label{eq:AB1}
 A_k U = (B_1-A_1)U\Lambda^k(I_n-\Lambda^k)^{-1}.
\end{equation}

\item [\emph{(c)}] $A_i U \sum\limits_{j=1}^i \Lambda^j= B_k U \Lambda^i \sum\limits_{j=1}^k \Lambda^j$, for any two positive integers $i$ and $k$.
\end{itemize}
\end{Theorem}
\begin{proof}
%
Clearly, (a) is true for $k=1$. Suppose that the statement is true for a positive integer $k=s$; that is,
\begin{align*}
A_1 U= B_s U \sum\limits_{j=1}^s \Lambda^j.
\end{align*}
We notice that
\begin{align*}
A_1U-B_s\Delta_{1,s}A_1U &=(A_1+B_s)\Delta_{1,s}A_1U-B_s\Delta_{1,s}A_1U\\
&= B_s\Delta_{1,s}A_1U\sum\limits_{j=1}^s \Lambda^j =B_s\Delta_{1,s}B_1 U \sum\limits_{j=2}^{s+1} \Lambda^j,
\end{align*}
so that
\begin{align*}
&A_1U=B_s\Delta_{1,s}B_1U\Lambda+B_s\Delta_{1,s}B_1 U \sum\limits_{j=2}^{s+1} \Lambda^j=B_{s+1}U\sum\limits_{j=1}^{s+1} \Lambda^j.
\end{align*}

The result of the first part of (b) has been given in our introduction. We thus omit the proof here.
Since
 \begin{equation*}
  A_k U = B_k U \Lambda^k=(A_k+B_1-A_1)U \Lambda^k= A_k U \Lambda^k + (B_1-A_1) U\Lambda^k,
\end{equation*}
we see that~\eqref{eq:AB1} holds, while $1\not\in \lambda(\Lambda)$.  Here, the second equality follows from Corollary~\ref{cor:dif}.

To prove (c), we first show that for any positive integer $i$,
\begin{equation*}
A_1 U  \Lambda^i = A_i U \sum\limits_{j=1}^i \Lambda^j.
\end{equation*}
By Theorem~\ref{thm:trans1}, since
$A_i = A_1 - B_1 \Delta_{i-1,1} A_1$ and
$B_i = B_{1}  \Delta_{i-1,1} B_{i-1}$,
we have
\begin{equation*}
(A_1- A_i) U = (B_1 \Delta_{i-1,1} A_1) U =   B_1 \Delta_{i-1,1} B_{i-1} U\sum_{j=1}^{i-1}\Lambda^j=   B_i U\sum_{j=1}^{i-1}\Lambda^j.
\end{equation*}
Or, equivalently,
\begin{equation*}
A_1 U \Lambda^i = A_i U \Lambda^i+B_i U \Lambda^i  \sum_{j=1}^{i-1}\Lambda^j 
= A_i U \sum_{j=1}^{i}\Lambda^j,
\end{equation*}
since  $A_i U= B_i U  \Lambda^i$.

Second, from (a), we have already proved (c)
for $i=1$ and a given positive integer $k$.
Assume (c) is true for $i = s$; that is,
\[
A_s U \sum_{j=1}^s \Lambda^j = B_k U\Lambda^s \sum_{j=1}^k \Lambda^j.
\]
Then
\begin{align*}
A_{s+1}U\sum_{j=1}^{s+1}\Lambda^j &= A_1 U \Lambda^{s+1}= (A_s U \sum_{j=1}^s \Lambda^j) \Lambda= B_k U \Lambda^{s+1} \sum_{j=1}^{k} \Lambda^j.
\end{align*}

\end{proof}

\section{Modified  AB-Algorithm }

Let {$\{A_k-\lambda B_k\}$} be the sequence of matrices generated by Algorithm~\ref{abalg}. {Before we move on, we should emphasize that the structure of the matrix pencil $A_k-\lambda B_k$ is invariant once the subscripts $i+j = k$; that is, the generation of the sequence $\{A_k-\lambda B_k\}$ is independent of  the subscript in $A_i$, $A_j$, $B_i$ and $B_j$. To fully take advantages of this invariance, we would like to design algorithms by applying Theorem~\ref{thm:trans1} to generate accelerated iterations with convergence of any desired order as follows.}

%
%
%
%

{
\begin{Algorithm}\label{aa2}
{\emph{(Modified AB-Algorithm)}}
\begin{enumerate}

\item {Given a positive integer $r>1$,
a tolerance  $\tau>0$, and a positive integer $kmax$, let $(\widehat{A}_1,\widehat{B}_1)=(A_1,B_1)$;}

\item    {For} $k= 2,\ldots,$ iterate {until
{$\mbox{dist}(\mbox{Null}(\widehat{A}_{k-1}),\mbox{Null}(\widehat{A}_{k}))< \tau$}  or $k> kmax$.}

 \begin{align*}
\widehat{A}_{k}& =A_{k-1}^{(r-1)}(A_{k-1}^{(r-1)}+\widehat{B}_{k-1})^{-1}\widehat{A}_{k-1},\\
\widehat{B}_{k}& =\widehat{B}_{k-1}(A_{k-1}^{(r-1)}+\widehat{B}_{k-1})^{-1}B_{k-1}^{(r-1)},
\end{align*}
    until convergence, where $(A_{k-1}^{(r-1)},B_{k-1}^{(r-1)})$ is defined in step 3.
\item
     {For} $\ell=1,\ldots,r-2$, iterate
   \begin{align*}
A_{k-1}^{(\ell+1)}& =A_{k-1}^{(\ell)}(A_{k-1}^{(\ell)}+\widehat{B}_{k-1})^{-1}\widehat{A}_{k-1},\\
B_{k-1}^{(\ell+1)}& =\widehat{B}_{k-1}(A_{k-1}^{(\ell)}+\widehat{B}_{k-1})^{-1}B_{k-1}^{(\ell)},
\end{align*}
with $(A_{k-1}^{(1)},B_{k-1}^{(1)})=(\widehat{A}_{k-1},\widehat{B}_{k-1})$.
 \end{enumerate}
\end{Algorithm}
}
 For clarity, a thing should be emphasized here. The AB algorithm has been developed to obtain the stable deflating subspace of the generalized eigenvalue problem $A_1U =  B_1 U\Lambda$. However,  {the sequence $\{A_k U\}$ provided in Algorithm~\ref{abalg}  converges only r-linearly to $0$, once the spectral radius of $\Lambda$ is less than $1$, and the sequence $\{B_k\}$  is uniformly bounded.}
From Algorithm~\ref{aa2} it follows that
\begin{subequations}
\begin{eqnarray}\label{eq:conv}
&&{A}_{k-1}^{(\ell+1)}  U = {B}_{k-1}^{(\ell+1)} U \Lambda^{(\ell+1)r^{k-2}},\\
&&\widehat{A}_{k} U =\widehat{B}_{k} U \Lambda^{r^{k-1}},
\end{eqnarray}
\end{subequations}
for $k = 2,\ldots$, and $\ell = 1,\ldots, r-2$,
and  $(\widehat{A}_{k},\widehat{B}_{k}) =
(A_{r^{k-1}},B_{r^{k-1}})$.
It follows from Theorem~\ref{thm:trans2} that
\begin{align*}
\|\widehat{A}_{k} U\|\leq
\dfrac{\|(B_1-A_1)U\|}{1-\|\Lambda\|^{r^{k-1}}}
\|\Lambda\|^{r^{k-1}},
\end{align*}
where $\|.\|$ is a matrix induced norm such that $\|\Lambda\|<1$. Thus the sequence $\{\widehat{A}_{k} U\}$ converges to $0$ with r-order $r$. For a full account of the definition of the rate of convergence, the reader is referred to~\cite{Kelly1995}.

Note that given two initial $n\times n$ matrices $A_1$ and $B_1$,
 the overall cost for computing the modified  AB-algorithm per iteration is $\frac{14(r-1)}{3} n^3$ flops.
The computation cost of the modified AB-algorithm with positive integer $r$
definitely increases as $r$ increases. Theoretically, Algorithm~\ref{aa2} provide a $r$-order convergence
sequence which approximates the solution of the stable subspace of $A-\lambda B$. Numerically, if $\rho(\Lambda)$ is not sufficiently close to $1$, choosing $r=2$ will be fast enough.

\section{Application of the AB-Algorithm for Solving the Matrix Square Root}
 { We notice that only recently, the modified AB-Algorithm with $r=2$ have been adjusted specifically for solving a kind of Sylvester matrix equations~\cite{Lin20152171} and the palindromic generalized eigenvalue problem~\cite{LI20112269}. In this section, we show that the AB-algorithm provide an alternative way to
compute the matrix square root. In particular, the speed of convergence of the AB-algorithm can be of any desired order. As mentioned before, numerical methods for solving the matrix square root are numerous. Comparison of numerical performance among different methods is something worthy of our investigation and is in process.}
In~\eqref{eq:conv} we see that the sequence $\{\widehat{A}_{i} U\}$ converges with r-order $r$ to $0$. We then in this section use this accelerated techniques  to solve the quadratic matrix equation defined in~\eqref{eq:quad},
i.e., find the principle square root $\sqrt{S}$ of the matrix $S$ with $\lambda(S) \subset \mathbb{C}^+$. To this end, we relate~\eqref{eq:quad} to the generalized eigenvalue problem
{
 \begin{align}\label{eq:space}
 A \bb I_{n} \\ \sqrt{S} \eb= B \bb I_{n} \\ \sqrt{S} \eb \sqrt{S},
\end{align}
}
where $A=\bb 0 & I_{n} \\ S & 0 \eb$ and $B=I_{2n}$. Since $\lambda(S) \subseteq \mathbb{C}^+$, there is no guarantee that the AB-algorithm will converge. To remedy this situation, this matrix $\sqrt{S}$ in~\eqref{eq:space} must be retreated. One way is to apply the M\"{o}bius transformation
{
\begin{equation*}
\mathcal{C}_{\sqrt{S}}(\gamma I_n) = (\gamma I_n-\sqrt{S})(\gamma I_n+\sqrt{S})^{-1},
\end{equation*}
}
where $\gamma > 0$ and {$-1\not\in\lambda(\gamma I_n-\lambda \sqrt{S})$}, i.e., $-1$ is not an eigenvalue of the matrix pencil {$\gamma I_n-\lambda\sqrt{S}$}; that is, recast~\eqref{eq:space} in the following equation
\begin{equation}\label{eq:genAB}
A_1\bb I_{n} \\ \sqrt{S} \eb=B_1\bb I_{n} \\ \sqrt{S} \eb \mathcal{C}_{\sqrt{S}}(\gamma I_n) ,
\end{equation}
where $A_1=\gamma B-A$ and $B_1=\gamma B+A$. Observe that $\rho(\mathcal{C}_{\sqrt{S}}(\gamma I_n)) < 1$ since $\lambda(S)\subseteq \mathbb{C}^+$. Upon using the AB-algorithm, it can be easily checked that
for any integer $k\geq1$, $A_k$ and $B_k$ can be expressed as
\begin{equation}\label{eq:abk}
A_k=\bb Q_k & -I_n\\ -S & Q_k \eb,\,B_k=\bb Q_k & I_n\\ S & Q_k \eb,\end{equation}
respectively, where the sequence $\{Q_k\}$ satisfies
 $Q_i Q_j=Q_j Q_i$, for any integers $i, j > 0$, and
 the following iteration
\begin{align}\label{fix}
Q_{k+1}=(\gamma Q_k+S)(\gamma I_n+Q_k)^{-1}
\end{align}
with $Q_1=\gamma I_n$. Note that once $Q_k = \sqrt{S}$
for some $k$,  it follows that $Q_\ell = \sqrt{S}$ for all $\ell \geq k$.

Specifically, let $\mathcal{C}_\gamma(\lambda)=\frac{\gamma-\lambda}{\lambda+\gamma}$ be the M\"{o}bius transformation with a parameter $\gamma \neq 0$ and $\lambda \neq -\gamma$. Then, the inverse scalar M\"{o}bius transformation can be written as
\begin{equation*}
\mathcal{C}_\gamma^{-1}(\lambda)=\gamma\frac{1-\lambda}{1+\lambda},\quad \lambda\neq -1.
\end{equation*}
Let $\lambda=e^{\frac{2j \pi i}{n}}\in S_n \backslash \{-1\}$, where $1\leq j < n$.
It follows that the real part of the square of $a:=\mathcal{C}_\gamma^{-1}(\lambda)$ is a real negative number, since
\begin{equation}\label{eq:reala}
a^2 = \gamma^2 (\frac{1-e^{\frac{2j\pi i}{n}}}{1+e^{\frac{2j\pi i}{n}}})^2
=
-\gamma^2\tan^2(\frac{j\pi}{n})<0.
\end{equation}

From~\eqref{eq:reala} and Theorem~\eqref{thm:well-defined}, it follows that the AB-algorithm will terminate prematurely only if
$\lambda(S)\subseteq\mathbb{C}^-$; that is,  once $\lambda(S) \subseteq \mathbb{C}^+
$, or even, $\lambda(S) \subseteq \mathbb{C}^+ \cup\{0\}$,
the sequence of matrix pencils $\{A_k-\lambda B_k\}$, initiated by~\eqref{eq:genAB}, is well-defined.

With an eye on the structure of the matrix pencil $A_k-\lambda B_k$, we look for
an accelerated iteration induced by the assumption of $\widehat{A}_{k}-\lambda
\widehat{B}_{k}$ in Algorithm~\ref{aa2}.

{
\begin{Algorithm} \label{square}
{\emph{(Iteration for solving the matrix square root)}}
\begin{enumerate}

\item
 {Given a positive integer $r>1$,
a tolerance  $\tau>0$,
 and a positive integer $kmax$,}
let $\widehat{Q}_1={Q}_1=\gamma I_n$;
\item    {For} $i=2,\ldots$, {iterate} {until
{$\|\widehat{Q}_{k}-\sqrt{S}\|< \tau$}  or $k>kmax$.}
 \begin{align*}
\widehat{Q}_{k}&:={(S+\widehat{Q}_{k-1}{Q}_{k-1}^{(r-1)})}{(\widehat{Q}_{k-1}+{Q}_{k-1}^{(r-1)})^{-1}},
\end{align*}
    until convergence, where $\widehat{Q}_{k-1}^{(r-1)}$ is defined in step 3.
\item
     {For} $\ell=1,\cdots,r-2$, iterate
 \begin{align*}
{Q}_{k-1}^{(\ell+1)}&:={(S+\widehat{Q}_{k-1}{Q}_{k-1}^{(\ell)})}{(\widehat{Q}_{k-1}+{Q}_{k-1}^{(\ell)})^{-1}},
\end{align*}
with $\widehat{Q}_{k-1}^{(1)}=\widehat{Q}_{k-1}$.
 \end{enumerate}
\end{Algorithm}

}

Note that $\widehat{Q}_{k} = Q_{r^{k-1}}$ for  $k \geq 1$, and with the assumption of the existence of iterative sequences, we immediately have the following iterative formulae. We omit the proof here because the result can be straightforwardly shown by using induction.
\begin{Theorem}
Assume that the sequences generated by Algorithm~\ref{square} can be constructed with no break down. Then, we have the following two iterative formulae.
\begin{enumerate}
  \item When $r$ is even, let $q = \frac{r}{2}$. We have
  \begin{align}\label{eq:whatQ1}
  \widehat{Q}_{k+1}=(\sum\limits_{j=0}^{q} {r \choose 2j}\widehat{Q}_k^{r-2j} S^j)(\sum\limits_{j=0}^{q-1} {r \choose 2j+1}\widehat{Q}_k^{r-2j-1} S^{j})^{-1}.
  \end{align}
   \item While $r$ is odd, let $q = \frac{r-1}{2}$. We have
  \begin{align}\label{eq:whatQ2}
  \widehat{Q}_{k+1}=(\sum\limits_{j=0}^{q} {r \choose 2j}\widehat{Q}_k^{r-2j} S^j)(\sum\limits_{j=0}^{q} {r \choose 2j+1}\widehat{Q}_k^{r-2j-1} S^{j})^{-1},
  \end{align}
\end{enumerate}
{
where the notation ${n \choose k}$ denotes the number of $k$-combinations
from the set $S=\{1,2,\cdots,n\}$ of $n$ elements.}
\end{Theorem}
We notice that if $S$ is a nonsingular matrix, then~\eqref{eq:whatQ1} and~\eqref{eq:whatQ2} can be simply expressed by the following rule:
 \begin{align*}
  \widehat{Q}_{k+1}=V_m U_m^{-1},
  \end{align*}
where
{
\begin{eqnarray*}
&&V_m  =\sum\limits_{j=0}^{[\frac{m}{2}]} {m \choose 2j}\widehat{Q}_k^{m-2j} S^j=\dfrac{1}{2}((\widehat{Q}_k+\sqrt{S})^m+(\widehat{Q}_k-\sqrt{S})^m), \\ &&U_m=\sum\limits_{j=0}^{[\frac{m-1}{2}]} {m \choose 2j+1}\widehat{Q}_k^{m-2j-1} S^j=\dfrac{(
{\sqrt{S})^{-1}}}{2}((\widehat{Q}_k+\sqrt{S})^m-(\widehat{Q}_k-\sqrt{S})^m).
\end{eqnarray*}}

Importantly, under nonsingularity assumption,  a strong result related to the sequences
$\{ \mathcal{C}_{\sqrt{S}}(Q_i) \}$ and $\{ \mathcal{C}_{\sqrt{S}}(\widehat{Q}_i) \}$
hold.
{
\begin{Lemma}\label{lem:qk}
Suppose that $S$ is nonsingular. Let $i$, $j$, and $k$ be any positive integers, and $1\leq i,j \leq k$. Then the following properties hold.
\begin{enumerate}
\item
For the sequence~$\{{Q}_k\}$, we have
\par\noindent
\begin{itemize}
\item[a.]
${Q}_{k}=\sqrt{S}(I_n+\mathcal{C}_{\sqrt{S}}({Q}_{1})^{{k}})(I_n-\mathcal{C}_{\sqrt{S}}({Q}_{1})^{{k}})^{-1}$,
\item[b.]
{$\mathcal{C}_{\sqrt{S}}({Q}_{i})^{j}=\mathcal{C}_{\sqrt{S}}({Q}_{j})^{i}.$}
\end{itemize}
\item
For the sequence~$\{\widehat{Q}_k\}$, we have
\par\noindent
\begin{itemize}
\item[a.]
$\widehat{Q}_{k} =\sqrt{S}(I_n+\mathcal{C}_{\sqrt{S}}(\widehat{Q}_{1})^{r^{k-1}})(I_n-\mathcal{C}_{\sqrt{S}}(\widehat{Q}_{1})^{r^{k-1}})^{-1}$,
\item[b.]
{$\mathcal{C}_{\sqrt{S}}(\widehat{Q}_{i})^{r^{k-i}}=\mathcal{C}_{\sqrt{S}}(\widehat{Q}_{j})^{r^{k-j}}$.}
\end{itemize}
%
\end{enumerate}
\end{Lemma}
}
\begin{proof}
It follows from Theorem~\ref{thm:trans2} and $Q_1=\gamma I_n$ that
\begin{align}\label{eq:abk2}
{
{A}_k U(I_n-\mathcal{C}_{\sqrt{S}}({Q}_1)^{k})=(B_1-A_1)U\mathcal{C}_{\sqrt{S}}({Q}_1)^{k}.}
\end{align}
Then,~\eqref{eq:abk} and~\eqref{eq:abk2} yield
\begin{align}\label{eq:qk}
({Q}_k-\sqrt{S})(I_n-\mathcal{C}_{\sqrt{S}}({Q}_1)^{k})&=2 \sqrt{S} \mathcal{C}_{\sqrt{S}}({Q}_1)^{k}.
\end{align}
By adding $2 \sqrt{S}(I_n-\mathcal{C}_{\sqrt{S}}({Q}_1)^{k})$ to both sides of~\eqref{eq:qk}, we have
\begin{align}\label{eq:qk2}
({Q}_k+\sqrt{S})(I_n-\mathcal{C}_{\sqrt{S}}({Q}_1)^{k})&=2 \sqrt{S}.
\end{align}
From~\eqref{eq:qk} and~\eqref{eq:qk2} together, it must be that
{
\begin{align*}
{Q}_k(I_n-\mathcal{C}_{\sqrt{S}}({Q}_1)^{k}) =\sqrt{S} (I_n +\mathcal{C}_{\sqrt{S}}({Q}_1)^{k}).
\end{align*}
}
 Since $S$ is nonsingular, it follows that $1\not\in\lambda(\mathcal{C}_{\sqrt{S}}({Q}_1))$ so that
{
\begin{align*}
{Q}_k= \sqrt{S} (I_n +\mathcal{C}_{\sqrt{S}}({Q}_1)^{k})(I_n-\mathcal{C}_{\sqrt{S}}({Q}_1)^{k}) ^{-1},
\end{align*}
}
which is equivalent to
{
\begin{align*}
\mathcal{C}_{\sqrt{S}}({Q}_1)^{k} = \mathcal{C}_{\sqrt{S}}({Q}_k).
\end{align*}
}
Since $k$ is an arbitrary positive integer, we have
{
\begin{align*}
 (\mathcal{C}_{\sqrt{S}}({Q}_i))^{j}
 = \mathcal{C}_{\sqrt{S}}({Q}_1)^{ij} =  (\mathcal{C}_{\sqrt{S}}({Q}_j))^{i},
\end{align*}
}
for $1\leq i, j \leq k$.

Also, by Theorem~\ref{thm:trans2}, Algorithm~\ref{aa2},
$\widehat{Q}_1 = \gamma I_n$, we have
\begin{align*}
\widehat{A}_k U=\widehat{B}_k U (\mathcal{C}_{\sqrt{S}}(\widehat{Q}_1))^{r^{k-1}},
\end{align*}
which then completes the proof of {
part 2a. and part 2b}. by applying the same strategies as above.
\end{proof}


{
Indeed, this iteration in Algorithm~\ref{square} converges to $\sqrt{S}$
with {
q}-order $r$.
}
\begin{Theorem}\label{thm:qua}
Suppose that $S$ is a nonsingular matrix. Let $\|.\|$ be a matrix induced norm such that $\|\mathcal{C}_{\sqrt{S}}(\widehat{Q}_1)\|<1$. Then,
{
\begin{equation*}
\|\widehat{Q}_{k+1}-\sqrt{S}\| \leq \mu \|\widehat{Q}_k-\sqrt{S}\|^r,
\end{equation*}
}
for some $\mu > 0$; that is, $\widehat{Q}_k\rightarrow \sqrt{S}$ with {
q}-order $r$.
\end{Theorem}
\begin{proof}

Using~\eqref{eq:qk} and $\widehat{Q}_k = {Q}_{r^{k-1}}$, we see that
\begin{align*}
\widehat{Q}_k-\sqrt{S}&=2\sqrt{S}\mathcal{C}_{\sqrt{S}}({Q}_1)^{r^{k-1}}(I_n-\mathcal{C}_{\sqrt{S}}({Q}_1)^{r^{k-1}})^{-1}.
\end{align*}
Without loss of generality we assume that $\widehat{Q}_{k}\neq \sqrt{S}$ for all $k$. Otherwise, $\widehat{Q}_{\ell}= \sqrt{S}$ for all $\ell\geq k$.
It follows that
\begin{align*}
\dfrac{\|\widehat{Q}_{k+1}-\sqrt{S}\|}{\|\widehat{Q}_k-\sqrt{S}\|^r}&\leq \dfrac{\|2\sqrt{S}\mathcal{C}_{\sqrt{S}}(Q_1)^{r^{k}}\| \|I_n-\mathcal{C}_{\sqrt{S}}(Q_1)^{r^{k-1}}\|^r}{\|2\sqrt{S}\mathcal{C}_{\sqrt{S}}(Q_1)^{r^{k-1}}\|^r(1-\|\mathcal{C}_{\sqrt{S}}(Q_1)^{r^{k}}\|)}\\
&\leq \dfrac{2\|\sqrt{S}\|\|\mathcal{C}_{\sqrt{S}}(Q_1)^{r^{k-1}}\|^r
\|I_n-\mathcal{C}_{\sqrt{S}}(Q_1)^{r^{k-1}}\|^r
}{2^{r}\frac{\|\mathcal{C}_{\sqrt{S}}(Q_1)^{r^{k-1}}\|^r}{\|(\sqrt{S})^{-1}\|^r}
(1-\|\mathcal{C}_{\sqrt{S}}(Q_1)^{r^{k}}\|)
}
\\
&\leq2^{1-r}\|\sqrt{S}\|  \|(\sqrt{S})^{-1}\|^r\sup\limits_{k\geq 1}\dfrac{(1+\|\mathcal{C}_{\sqrt{S}}(Q_1)\|^{r^{k-1}})^r}{1-\|\mathcal{C}_{\sqrt{S}}(Q_1)\|^{r^{k}}}\\
&\leq  \mu:= \dfrac{2\|\sqrt{S} \|}{
1-\|\mathcal{C}_{\sqrt{S}}(Q_1)\|^{r}}{  \|(\sqrt{S})^{-1}\|^r}
<\infty.
\end{align*}
%
\end{proof}
Note that for $r=2$ the iteration $\widehat{Q}_{k+1}=(\widehat{Q}_{k}^2+S)(2\widehat{Q}_{k})^{-1}=\frac{1}{2}(\widehat{Q}_{k}+S\widehat{Q}_{k}^{-1})$ with initial $\widehat{Q}_{1}=\gamma I_n$,
which is equivalent to the Newton's method for solving the matrix square root~\cite{Highambook08}, converges to $\sqrt{S}$ with quadratic convergence.
%
For $r=3$ we have  $\widehat{Q}_{k+1}=(\widehat{Q}_{k}^3+3\widehat{Q}_{k}S)(3\widehat{Q}_{k}^2+S)^{-1}$, which provides that a cubically convergent iteration converges to $\sqrt{S}$ with initial $\widehat{Q}_{1}=\gamma I_n$. Similarly, by Algorithm~\ref{square} we can make $\widehat{Q}_{k+1}$ converges to $\sqrt{S}$ $q$-superlinearly with any desired $q$-order $r$.
However, without the accelerated technique, we can show in the following that the original sequence $\{Q_k\}$ only converges to $\sqrt{S}$ q-linearly.
\begin{Theorem}
Suppose that $S$ is a nonsingular matrix. Let $\|.\|$ be a matrix induced norm such that $\|\mathcal{C}_{\sqrt{S}}(\widehat{Q}_1)\|<1$. Then,
{
\begin{equation*}
\|{Q}_{k+1}-\sqrt{S}\| \leq \mu \|{Q}_k-\sqrt{S}\|,
\end{equation*}
}
for some $\mu \in (0,1)$ and sufficient large $k$; that is, ${Q}_k\rightarrow \sqrt{S}$ q-linearly with $q$-factor $\mu$.
\end{Theorem}
\begin{proof}

From~\eqref{eq:qk}, we have
\begin{align*}
{Q}_k-\sqrt{S}&=2\sqrt{S}\mathcal{C}_{\sqrt{S}}({Q}_1)^{k}(I_n-\mathcal{C}_{\sqrt{S}}({Q}_1)^{k})^{-1}.
\end{align*}
Thus,
\begin{align*}
\|{Q}_{k+1}-\sqrt{S}\|&=
{
\|({Q}_k-\sqrt{S})(I_n-\mathcal{C}_{\sqrt{S}}(Q_1)^k)}
\mathcal{C}_{\sqrt{S}}(Q_1)(I_n-\mathcal{C}_{\sqrt{S}}(Q_1)^{k+1})^{-1}\|
\\
&\leq \|\mathcal{C}_{\sqrt{S}}(Q_1)\|
\dfrac{1+ \|\mathcal{C}_{\sqrt{S}}(Q_1)\|^{k}}{1-\|\mathcal{C}_{\sqrt{S}}(Q_1)\|^{k+1}} \|Q_k-\sqrt{S}\|.
\end{align*}
Since $\dfrac{1+ \|\mathcal{C}_{\sqrt{S}}(Q_1)\|^{k}}{1-\|\mathcal{C}_{\sqrt{S}}(Q_1)\|^{k+1}} \rightarrow 1$ as $k\rightarrow \infty$, there exists a constant $k_0$ such that
\begin{equation*}
\|\mathcal{C}_{\sqrt{S}}(Q_1)\| \dfrac{1+ \|\mathcal{C}_{\sqrt{S}}(Q_1)\|^{k}}{1-\|\mathcal{C}_{\sqrt{S}}(Q_1)\|^{k+1}} < 1
\end{equation*}
for {
$k  \geq k_0$.}
Let $\mu = \dfrac{1+ \|\mathcal{C}_{\sqrt{S}}(Q_1)\|^{k_0}}{1-\|\mathcal{C}_{\sqrt{S}}(Q_1)\|^{k_0+1}}\|\mathcal{C}_{\sqrt{S}}(Q_1)\|  $, which completes the proof.
%
\end{proof}

{
In the next result, we show that the AB-algorithm still converges, while solving the square root of a singular matrix, which is hard to be handled in general. See~\cite{Meini2004} for further discussion.
}
\begin{Corollary}\label{DAconvthm2}
Suppose that $S$ is a singular matrix having $\lambda(S) \subseteq \mathbb{C}^+ \cup\{0\}$ and the null eigenvalues are semisimple.
Then,
\begin{enumerate}
\item
${Q}_k \rightarrow \sqrt{S}$ {sublinearly},

\item
$\widehat{Q}_k \rightarrow \sqrt{S}$ q-linearly with q-factor $\frac{1}{r}$.
\end{enumerate}
\end{Corollary}
\begin{proof}
Let $P$ be an invertible matrix so that
$\mbox{diag}(0_p, J_{n-p})
= P \sqrt{S} P^{-1} $
be the Jordan canonical form of $\sqrt{S}$ with $\lambda(J_{n-p})\subset\mathbb{C}^+$. Upon the use of substitution and Lemma~\ref{lem:qk},  we have
{
\begin{eqnarray*}
P {Q}_k P^{-1}= \mbox{diag}({Q}_k^{(11)}, {Q}_k^{(22)}),
\end{eqnarray*}
}
where ${Q}_k^{(11)}$ is derived directly by~\eqref{fix} and
${Q}_k^{(22)}$ is followed from Lemma~\ref{lem:qk} such that
\begin{align*}
 {Q}_k^{(11)} &=\dfrac{\gamma I_n}{k},\\
 {Q}_k^{(22)} &=J_{n-p}(I+\mathcal{C}_{J_{n-p}}(\gamma I_{n-p})^{k})(I-\mathcal{C}_{J_{n-p}}(\gamma I_{n-p})^{k})^{-1}.
\end{align*}
Since $ \{{Q}_k^{(11)}\}$ converges to zero sublinearly
and $ \{{Q}_k^{(22)}\}$ converges to zero  q-linearly,  $\{{Q}_k\}$ converges to $\sqrt{S}$ sublinearly.
 In the similar way, we have
{
\begin{eqnarray*}
P \widehat{Q}_k P^{-1}= \mbox{diag}(\widehat{Q}_k^{(11)}, \widehat{Q}_k^{(22)}),
\end{eqnarray*}
}
where
\begin{align*}
\widehat{Q}_k^{(11)} &=\dfrac{\gamma I_n}{r^k},\\
\widehat{Q}_k^{(22)}  &= J(I_n+\mathcal{C}_{J}(\gamma I_{n})^{r^{k-1}})(I_n-\mathcal{C}_{J}(\gamma I_{n})^{r^{k-1}})^{-1}.
\end{align*}
Since $ \{\widehat{Q}_k^{(11)}\}$ converges to zero q-linearly with q-factor $r$, it follows that $\{\widehat{Q}_k\}$ converges to $\sqrt{S}$ q-linearly with {
q}-factor $r$.
\end{proof}
\begin{Remark}
Once the spectral radius of $\sqrt{S}$ in \eqref{eq:space} is not less than 1, we can apply the M\"{o}bius transformations to shift eigenvalues of $S$ such that $\rho(\mathcal{C}_{\sqrt{S}}(\gamma I_n))$ in~\eqref{eq:genAB} is less than 1.  %
%
We would like our $\gamma$ to have a capacity such that the optimal convergence speed  in~Algorithm~\ref{square} can be achieved. To this end, we seek $\gamma$ to be equal to the optimal solution $\gamma_0$ of the following min-max problem
\begin{align}\label{eq:gamma}
\gamma_0:=\min\limits_{\gamma>0}\max\limits_{\lambda\in \lambda(S)} |\frac{\sqrt{\lambda}-\gamma}{\sqrt{\lambda}+\gamma}|.
\end{align}
%
This min-max problem is also known as the ADI min-max problem~\cite{MR3076884}.
Numerical approaches for solving~\eqref{eq:gamma} are numerous.  Here
we will not discuss it further. The reader is referred to~\cite{MR2494951,MR3076884,MR1742324} for example.

\end{Remark}

\section{Concluding remarks}
{By computing left null spaces, the contribution of this work is twofold. Theoretically, it provides an iterative method, embedded with a discrete-type flow property, to solve the stable deflating subspace of a matrix pencil $A-\lambda B$. This property then allows us to advance the iterative method.  Numerically, we have discussed with the numerical behavior of the AB-algorithm, including both low computational cost and high numerical reliability. Since the solution of the matrix square root can be interpreted in terms of the stable deflating subspace of a matrix pencil, our method can be used to compute the matrix square root. We show that the speed of convergence has q-order $r$, and even more, for the singular case, where $S$ is singular having no negative real eigenvalues, and the null eigenvalues are semisimple, the iteration still succeeds with a linear rate of convergence.

Particularly, since Algorithm~\ref{square} corresponds to Newton iteration with $r=2$ and the initial guess $\gamma I_n$, the limiting accuracy should not be worse than $\kappa(\sqrt{S})\epsilon$, where $\kappa(\sqrt{S})$ is the condition number of $\sqrt{S}$ and $\epsilon$ is machine precision \cite[Table 6.2 on p.147]{Highambook08}. Numerically, it is known that a stable variant of Newton iteration, the IN iteration \cite[(6.20) on p.142]{Highambook08}, has been proposed with the limiting accuracy equal to $\epsilon$. Whether the AB-algorithm for $r=2$ has the desired accuracy or even more for $r> 2$ is something worthy of further investigation. Numerically, modified AB-algorithms for $r=2$ were also developed for solving generalized continuous/discrete-time
algebraic Riccati equations~\cite{LI20112269} and $\star$-Sylvester matrix equation~\cite{Lin20152171}. How to apply the accelerated techniques in the work for solving other matrix equations (for example, matrix $p$th root) leads to the work in future.
}

\section*{Acknowledgment}
This research work is partially supported by the Ministry of Science and Technology and the National Center for Theoretical Sciences in Taiwan.

%
\def\cprime{$'$}

\end{document}